\documentclass[11pt]{article}

\usepackage{bbm}
\usepackage{amsthm}
\usepackage{amssymb}
\usepackage{amsmath}
\usepackage{color}
\usepackage{graphics}
\usepackage{graphicx}
\usepackage{hyperref}
\usepackage{tikz}
\usepackage{hyperref}
\usepackage{subcaption}

\newtheorem{theorem}{Theorem}
\newtheorem{assumption}{Assumption}

\newtheorem{lemma}{Lemma}

\theoremstyle{definition}

\theoremstyle{remark}

\newcommand{\brac}[1]{\left( #1\right)}
\newcommand{\calig}[1]{ {\cal #1} }
\newcommand{\abs}[1]{\left| #1\right|}
\newcommand{\floor}[1]{\left\lfloor #1\right\rfloor}

\newcommand{\closedbrac}[1]{\left[ #1\right]}

\newcommand{\indicator}[1]{ \mathbbm{1}_{\left\{#1\right\}} }

\newcommand{\MVD}[0]{\mathbb{D}_{\calig{M}}}
\newcommand{\MVDup}[0]{\mathbb{D}_{\calig{M}}^\uparrow}

\newcommand{\MVC}[0]{\mathbb{C}_{\calig{M}}}

\newcommand{\Dup}[0]{\mathbb{D}^\uparrow}
\newcommand{\CCup}[0]{\mathbb{C}^\uparrow}
\newcommand{\MVCnoatoms}[0]{\mathbb{C}_{\calig{M}_0}}
\newcommand{\MVCupnoatoms}[0]{\mathbb{C}_{\calig{M}_0}^\uparrow}

\newcommand{\eps}{\varepsilon}

\newcommand{\C}{{\mathbb C}}
\newcommand{\D}{{\mathbb D}}

\newcommand{\N}{{\mathbb N}}

\newcommand{\R}{{\mathbb R}}

\newcommand{\calM}{{\cal M}}

\newcommand{\calS}{{\cal S}}

\newcommand{\skp}{\vspace{\baselineskip}}

\newcommand{\lt}{\left}
\newcommand{\rt}{\right}

\newcommand{\To}{\Rightarrow}

\newcommand{\noi}{\noindent}

\begin{document}

\title{Fluid Limits for Shortest Job First with Aging}
\author{Yonatan Shadmi\thanks{Viterbi Faculty of Electrical Engineering,
Technion -- Israel Institute of Technology, Haifa 32000, Israel}\\shdami@campus.technion.ac.il}
\date{}
\maketitle

\begin{abstract}
    We investigate fluid scaling of single server queueing systems under the shortest job first with aging (SJFA) scheduling policy. 
    We use the measure-valued Skorokhod map to characterize the fluid limit for SJFA queues with a general aging rule and establish convergence results to the fluid limit.
    We treat in detail examples of linear and exponential aging.
    \skp
    
    \noi{\bf AMS subject classifications:} 60K25, 60G57, 68M20
    
    \skp
    
    \noi{\bf Keywords:}
    measure-valued Skorokhod map,
    measure-valued processes,
    fluid limits, shortest job first, aging
    
\end{abstract}

\section{Introduction}

It is well known that prioritizing jobs by their size is optimal in the sense of minimizing the number of jobs in system at any point in time (\cite{Schrage68}, \cite{smith78}). 
This could be important in systems with limited buffer size. 
As the average number of jobs in the queue and the average waiting time are related by Little's law, minimizing the number of jobs in the buffer is equivalent to minimizing the average waiting time (\cite{schrage66},\cite{Little61}). 
However, these kind of policies may cause starvation and deny resource access from large sized jobs. 
To avoid this undesired phenomenon, more dynamic approaches have been suggested, that take into consideration the sojourn time as well as the job size.
In particular, the priority of a job is dictated initially by its size, and is updated as time elapses.
This is done in practice (\cite{bach1986}, \cite{Feitelson2011}) and is common in the literature of computer science (\cite{silberschatz05}, \cite{Behera12}, \cite{Dellamico14}), and is known as aging.
In this case, even large sized jobs will eventually have higher priority than smaller jobs with later arrival time. 
Scaling limits of such policies have not been treated before, even though both fluid and diffusion limits of SJF were extensively studied.
This paper analyzes and characterizes fluid limits of SJFA for the first time; the main tool is the measure valued Skorokhod map from \cite{atar18}.

Prominent examples of aging rules are linear (\cite{silberschatz05}), exponential (\cite{bach1986},\cite{Feitelson2011}), \emph{highest response ratio next} (HRRN) (\cite{Behera12}) and the \emph{fair sojourn protocol} (FSP) (\cite{Dellamico14}). 
Among the many aging rules that our analysis covers, the linear is the simplest one and most of the examples in this paper consider this rule. 
Though somewhat more cumbersome, the exponential rule is used in practice in UNIX (\cite{bach1986}, \cite{Feitelson2011}), which motivates us to study it in details as well.
HRRN and FSP are essentially different and are explained later.
The exponential aging rule in UNIX is implemented as follows (as described in \cite{Feitelson2011}, where other details and examples can be found).
One employs the rule
\begin{align*}
    pri=cpu\_use+base+nice
\end{align*}
where: $pri$ stands for the priority value of a job at the current time, and \textbf{lower priority value} corresponds to \textbf{higher priority}; 
$cpu\_use$ stands for recent CPU usage and is the term responsible of the aging.
This value increases linearly in time only when the process is running (discretely on every clock interrupt), and is divided by two every second;
$base$ is the base priority for user processes;
$nice$ is an optional additional term, if a user is willing to reduce its own priority and be nice to others.
In our analysis, we ignore the terms $base$ and $nice$, as well as the linear increase of $cpu\_use$ which is irrelevant in non-preemptive policies because it happens only while the process is running and not while it waits for service. 
Note that the updating rule of $cpu\_use$ is a discrete time version of exponential aging due to the division by two every second of the priority value.

The linear and exponential aging rules are the simplest among the above examples, as they belong to a class of deterministic aging rules with non-intersecting aging trajectories. 
Here, by aging trajectories we mean trajectories on the time-priority plane, describing how a job priority is updated with time.
This means that jobs do not exchange places in the priority ordering as time progresses. 
A generalization of these two rules is the class of aging rules that can be described through an ordinary differential equation.
In this paper, we consider this class of aging rules. 

Job size based priority systems is an active research topic in the operations research literature since the work of Schrage and Miller \cite{schrage66}, \cite{Schrage68} (see also \cite{smith78} for a simpler proof of optimaility). 
In particular, some research effort has been devoted to scaling limits of such policies.
Fluid limits of the preemptive version, the shortest remaining processing time policy, were shown to exist, be unique, and were characterized in \cite{Down09_fluid_limits}. 
To describe the evolution of the system, a measure valued process was used, one that counts the number of jobs in the queue with job size in some set. 

Diffusion scaling limits for \emph{shortest remaining processing time} (SRPT), the preemptive version of the SJF, appeared shortly after in \cite{Gromoll11} under heavy traffic assumption, again using measure valued processes as state descriptors. 
The measure counting the number of jobs in the queue with size in some set is shown to converge to a Dirac measure,  concentrated on the rightmost point of the support of the job size distribution. 
The magnitude of the limiting Dirac delta is a reflected Brownian motion divided by the largest possible job size (zero if the job size is unbounded). 
A followup of this work is \cite{puha15}, where non-standard diffusion limit was taken to show a generalized state space collapse- the queue length and the workload converge to the same process. 

In \cite{atar18}, Atar et al. prove convergence to the fluid limit of several priority queues, including SRPT and SJF. 
They introduce a measure valued Skorokhod map, and use it to map the arrival process and the server effort process to a measure valued state descriptor. This technique allows considering also time-varying arrival and service distributions. 

Most of the work on such policies studies SRPT, however, Gromoll and Keutel argue in \cite{Gromoll12} that, at least on fluid scale, the performance of SRPT and SJF is the same (they converge to the same fluid limit). 
See also \cite{Nuyens06}, \cite{Harchol-Balter03} and \cite{Bansal01} for more comparisons between SJF and SRPT.

Though the literature on SJF and SRPT and their asymptotic behaviour is rich, there is no mentioning of aging in this context. 
With aging, the system state changes not only when jobs arrive and depart, but also vary with time, as the priority values undergo aging.
These dynamics prevent a direct use of the measure valued Skorokhod map as in \cite{atar18}.

Our objective here is to expand the existing asymptotic analysis to include SJFA with deterministic aging and non-overlapping trajectories. 
The idea is to first identify a coordinate transformation that transforms the system into an equivalent system, and then use the measure valued Skorokhod map with measures defined on the new coordinate system. 
As an easy example, for linear aging, let $S_i$ and $\tau_i$ be the $i$-th job's size and arrival time, resp. 
Whereas SJFA prioritizes according to $S_i-\brac{t-\tau_i}$, this is equivalent to prioritizing according to $S_i+\tau_i$, which is independent of time. 
For a general aging rule, this transformation requires further effort.
With these means we are able to prove convergence to the fluid limit and characterize the limit.

Indeed, linear and exponential aging rules belong to the class of aging rules that are described by an ordinary differential equation, as mentioned above.
However, we have also mentioned the HRRN and the FSP, which are important policies that avoid starvation, but do not fall in the category we cover in this work.
In the HRRN policy, all jobs are assigned on arrival the same priority value, but this value increases with a rate which is the inverse of the job size. 
Thus, the aging trajectories do intersect, making it impossible to use our method. 
In the FSP, there is a virtual processor sharing server and jobs enter service in the order of their virtual departure process. 
In this setting, the aging trajectories are stochastic; they depend on the system state. 
This is a much more interesting and complicated situation which, we assume, requires much more sophisticated tools to analyze.
This motivates us to seek further tools to cover other settings.

The paper is organized as follows.
In \S2 we introduce the queueing model and state the main theorem.
The proof is given in \S3.
\S4 provides four examples with different arrival distributions and aging rules.

\subsection*{Notation}
For any Polish space $\calS$, let $\D_\calS$ denote the space of c\`adl\`ag functions from $\R_+$ to $\calS$, equipped with the Skorokhod J1 metric.
Let $\C$ denote its subspace of continuous functions.
If $\calS=\R_+$ we simply write $\D$ and $\C$.
Let $\Dup$ denote the subspace of non-decreasing elements of $\D$.
Let $\calM$ denote the space of finite Borel measures on $\R$ equipped with the L\'evy metric, which induces the topology of weak convergence, and let $\calM_0$ denote its subspace of atomless Borel measures on $\R$. 
It is well known that $\calM$ is a Polish space, and so is $\MVD$.
Let $\MVDup$ denote the subspace of non-decreasing elements of $\MVD$, where an element $\nu$ of $\MVD$ is said to be non-decreasing if $t\mapsto\int fd\nu_t$ is non-decreasing for every non-negative continuous and bounded function $f$.

\section{SJFA}
\subsection{Queueing model}\label{sec:Queueing model}
Consider a sequence of single server queues, indexed by $N\in\N$, operating under the SJFA with the same aging policy. 
Whenever the server is available, it admits the highest priority job into service, where the priority of each job is as follows. 
The system keeps track of each job's priority value, where lower value means higher priority. 
Each job is assigned an initial value on arrival which is the job size.
The value, $g$, varies smoothly with time according to an aging rule.
To formalize this, we order the jobs according to their arrival, and denote by $\tau_i^N$ and $S_i^N$ the arrival time and job size, resp., of the $i$-th job to enter the $N$-th system. 
The fact that the initial priority value is the job size is expressed by 
\begin{align}\label{eq:inital cond}
    g\brac{\tau_i^N}=S_i^N.    
\end{align}
The aging rule is modeled by the differential equation 
\begin{align}\label{eq:aging rule}
    \frac{dg}{ds}=f\brac{g,s},\quad s\geq\tau_i^N.
\end{align}
In the context of aging, $f$ is assumed to be negative, as the priority of a job increases as times progresses (i.e. $g$ is decreasing).
The linear and exponential aging rules correspond to the choices $f\brac{g,s}=-c$ and $f\brac{g,s}=-\lambda g$, resp., for constants $c,\lambda>0$. 

In fact, under certain conditions, Equation \eqref{eq:aging rule} can be solved backwards as well to obtain a solution for all $s$ in $\R_+$.
This fact is used is this work and, considering Equations \eqref{eq:inital cond} and \eqref{eq:aging rule} jointly, the aging trajectory of the $i$-th job is described by
\begin{align}\label{eq:aging+i.c.}
    \begin{cases}
    \frac{d}{ds}g=f\brac{g,s},\quad s\in\R_+,\\
    g\brac{\tau_i^N}=S_i^N.    
    \end{cases}
\end{align}

We require that $f$ is Lipschitz in the first argument.
Then this problem admits a unique solution on $\R_+$ (\cite[Theorem 6]{birkhoff1989}), that is, a unique aging trajectory $g:\R_+\to\R$ satisfying $g\brac{\tau_i^N}=S_i^N$. 
This way, any point $\brac{x,t}\in\R_+^2$ has exactly one trajectory passing through it, denoted $g_{\brac{x,t}}$. 
In particular, $g_i^N:=g_{\brac{S_i^N,\tau_i^N}}$ is the unique function $h:\R_+\to\R$ such that $\frac{d}{ds}h=f\brac{h(s),s}$, $s\in\R_+$, and $h\brac{\tau_i^N}=S_i^N$.
The trajectories of the linear and exponential aging rules are, resp., $g_{\brac{x,t}}\brac{s}=x-c\brac{s-t}$ and $g_{\brac{x,t}}\brac{s}=x\exp\brac{-\lambda\brac{s-t}}$ (see Figure \ref{fig:aging trajectories}).
\begin{figure}
    \centering
    \includegraphics[scale=0.5]{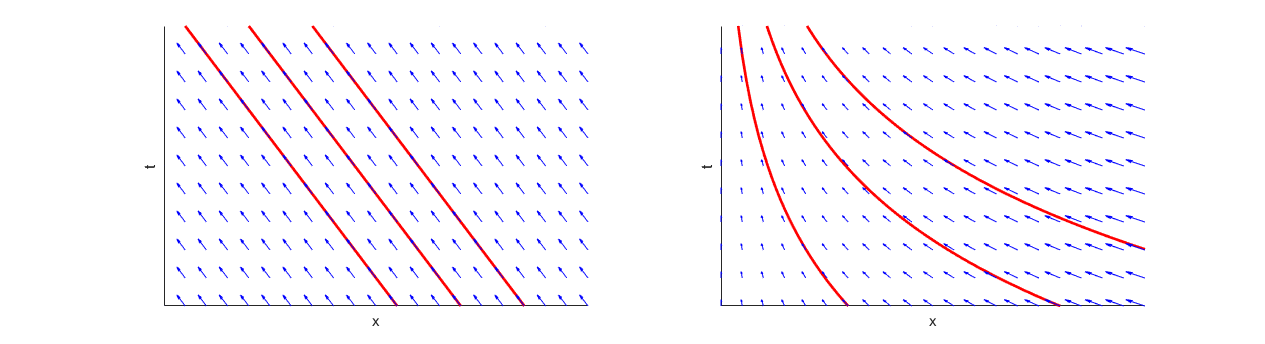}
    \caption{Illustration of aging trajectories in the cases of linear aging (left hand side) and exponential aging (right hand side). The arrows illustrate the flow of a priority value on the time-priority plane. The lines describe the trajectories that the priority values follow.}
    \label{fig:aging trajectories}
\end{figure}
We denote by $\theta_i^N$ the departure time of the $i$-th job, but stress that in our model the priority value continues to change according to the aging rule even after the job has departed from the system.

We introduce the following processes, underlying the queueing model in this section.
\begin{alignat}{3}\label{eq:processes}
    &\alpha^N &&:\quad&&\alpha_t^N\lt(-\infty,x\rt]\text{ is the work to have arrived at the queue by time $t$ }\\
    & && &&\text{with priority value in $\lt(-\infty,x\rt]$ at time $t$},\nonumber\\
    &\mu^N &&:&&\mu^N\brac{t}\text{ is the total work the server can complete by time $t$, if it is never idle},\nonumber\\
    &\xi^N &&:&&\xi_t^N\lt(-\infty,x\rt]\text{ is the work in the queue at time $t$ with priority value in $\lt(-\infty,x\rt]$ at time $t$},\nonumber\\
    &\beta^N &&:&&\beta_t^N\lt(-\infty,x\rt]\text{ is the work completed by the server by time $t$, with priority value in $\lt(-\infty,x\rt]$} \nonumber\\
    & && && \text{at time $t$}.\nonumber
\end{alignat}
The sample paths of $\alpha^N$, $\xi^N$ and $\beta^N$ are in $\MVD$, and $\mu^N$ has sample paths in $\Dup$.
Mathematically, $\alpha^N,\beta^N$ and $\xi^N$ are given by
\begin{align}
    &\alpha_t^N\lt(-\infty,x\rt] = \sum_{i=1}^\infty S_i^N\indicator{\tau_i^N\leq t}\indicator{g_i^N\brac{t}\leq x}\label{eq:alphaN}\\
    &\beta_t^N\lt(-\infty,x\rt] = \sum_{i=1}^\infty S_i^N\indicator{\theta_i^N\leq t}\indicator{g_i^N\brac{t}\leq x},\label{eq:betaN}\\
    &\xi_t^N\lt(-\infty,x\rt]=\alpha_t^N\lt(-\infty,x\rt]-\beta_t^N\lt(-\infty,x\rt]=\sum_{i=1}^\infty S_i^N\indicator{\tau_i^N\leq t< \theta_i^N}\indicator{g_i^N\brac{t}\leq x}.\label{eq:xiN}
\end{align}

The initial state of the buffer is already contained in $\alpha_0^N$. 
We assume that $\mu$ is given by $\mu\brac{t}=\int_0^tm^N\brac{s}ds$, where $m^N\brac{s}$ is the service rate at time $s$ satisfying $m^N\brac{s}\geq m_0>0$ for all $s$. 
Define the busyness process $B^N\brac{t}=\indicator{\text{The server is busy at time $t$}}$; then the work done by the server by time $t$ is $T^N\brac{t}=\int_0^tm^N\brac{s}B^N\brac{s}ds$, and $\iota^N=\mu^N-T^N$ is the lost work due to idleness.
Let the process $J\brac{t}$ denote the residual work of the job in service at time $t$. It then holds that
\begin{align*}
    \beta_t^N\lt[0,\infty\rt)=T^N\brac{t}+J^N\brac{t}-J^N\brac{0}.
\end{align*}

At any time when the server is ready to serve a new job, the scheduler admits to service the job with the highest priority, i.e. with the smallest priority value. 
If the queue is empty and the server is idle, the next job to arrive will be immediately admitted.

\subsection{Main result}
We provide sufficient conditions for $\brac{\bar{\xi}^N,\bar{\beta}^N,\bar{\iota}^N}$ to weakly converge as a result, and characterize this weak limit by a formula.

\begin{assumption}\label{ass:exist unique cond}
The function $f$ in \eqref{eq:aging+i.c.} is continuous as a function of two variables, and is globally Lipschitz in the first argument on $\R_+^2$, i.e. satisfies 
\begin{align*}
    \abs{f\brac{x_1,t}-f\brac{x_2,t}}\leq L\abs{x_1-x_2}
\end{align*}
for some $L>0$ for all $\brac{x_1,t}$ and $\brac{x_2,t}$ in $\R\times\R_+$. 
\end{assumption}
Assumption \ref{ass:exist unique cond} guaranties the existence of a unique path, $g\brac{s}$, that solves the ODE (\cite[Theorem 6]{birkhoff1989})
\begin{align*}
    &\begin{cases}
    \frac{\partial}{\partial s}g=f\brac{g,s},\quad s\in\R_+,\\
    g\brac{t}=x,
    \end{cases}
\end{align*}
for any given $x$ and $t$.
The uniqueness of the path is on $\R\times\R_+$ in the sense that there is only one solution passing through any point $\brac{x,t}\in\R\times\R_+$. 
As a consequence, two solution paths are either the same path, or do not intersect.

\begin{theorem}\label{thm:Fluid limit}
Let the aging rule be as in \eqref{eq:aging+i.c.}. Suppose that Assumption \ref{ass:exist unique cond} holds and that $\brac{\Bar{\alpha}^N,\Bar{\mu}^N}\To\brac{\alpha,\mu}\in\MVCupnoatoms\times\CCup$. 
Denote $\Xi\brac{t,x}:=\alpha_t\lt(-\infty,g_{\brac{x,0}}\brac{t}\rt]$.
Then $\brac{\bar{\xi}^N,\bar{\beta}^N,\bar{\iota}^N}\To\brac{\xi,\beta,\iota}$, where for all $x$ and $t>0$
\begin{align}
    &\beta_t\brac{x,\infty}+\iota\brac{t}=-\inf_{s\in\closedbrac{0,t}}\lt\{\brac{\Xi\brac{s,g_{\brac{x,t}}\brac{0}}-\mu\brac{s}}\land 0\rt\},\label{eq:lim comp1}\\
    &\xi_t\lt(-\infty,x\rt]=\Xi\brac{t,g_{\brac{x,t}}\brac{0}}-\mu\brac{t}-\inf_{s\in\closedbrac{0,t}}\lt\{\brac{\Xi\brac{s,g_{\brac{x,t}}\brac{0}}-\mu\brac{s}}\land 0\rt\}.\label{eq:lim comp2}
\end{align}
\end{theorem}

Above, \eqref{eq:lim comp2} characterizes the measure valued process $\xi$, and \eqref{eq:lim comp1} characterizes both $\iota$ and $\beta$ by taking the limit $x\to\infty$ and then subtracting $\iota(t)$.

In the proof given in the next section, the characterization of the limit is expressed by other means. 
The somewhat complicated notation above is used in order to avoid introducing a transformation of coordinates.

\section{Proof}

The fluid model is the unique triplet $\brac{\xi',\beta',\iota}\in\MVD\times\MVDup\times\Dup$ satisfying for given data $\brac{\alpha',\mu}\in\MVDup\times\Dup$ the following relations.
\begin{enumerate}
    \item $\xi'\closedbrac{0,x}=\alpha'\closedbrac{0,x}-\mu+\beta'\brac{x,\infty}+\iota$,
    \item $\xi'\closedbrac{0,x}=0$ $d\beta'\brac{x,\infty}-a.s.$,
    \item $\xi'\closedbrac{0,x}=0$ $d\iota-a.s.$,
    \item $\beta'\lt[0,\infty\rt)+\iota=\mu$.
\end{enumerate}
This was introduced as the measure-valued Skorokhod problem (MVSP) in \cite{atar18},  followed by an existence and uniqueness theorem (\cite[Proposition  2.8]{atar18}).

Recall that the sample paths of $\alpha^N$ and $\beta^N$ belong to $\MVD$ but not to $\MVDup$, so it is not yet clear how the MVSP is relevant to our model. 

The MVSP defines a map $\brac{\xi',\beta',\iota}=\Theta\brac{\alpha',\mu}$, referred to as the measure-valued Skorokhod map (MVSM). 
Namely, for each $x$, $\brac{\xi'\closedbrac{0,x},\beta'\brac{x,\infty}+\iota}=\Gamma\brac{\alpha'\closedbrac{0,x}-\mu}$, where $\Gamma:\D\to\D_+\times\Dup$, $\Gamma(\psi)=\brac{\Gamma_1(\psi),\Gamma_2(\psi)}$, is the Skorokhod map 
\begin{align*}
    &\Gamma_1\brac{\psi}\brac{t}=\psi\brac{t}-\inf_{s\in\closedbrac{0,t}}\brac{\psi\brac{s}\land 0}, & \Gamma_2\brac{\psi}\brac{t}=-\inf_{s\in\closedbrac{0,t}}\brac{\psi\brac{s}\land 0}.
\end{align*}
Note that \eqref{eq:lim comp1} and \eqref{eq:lim comp2} have this form exactly with $\psi=\Xi\brac{\cdot,g_{\brac{x,t}}\brac{0}}-\mu$.

\begin{proof}[Proof of Theorem \ref{thm:Fluid limit}]

The proof opens with the definition of a coordinate transformation, $\varphi:\R\times\R_+\to\R\times\R_+$, that maps a point $\brac{x,t}\in\R\times\R_+$ to a point $\brac{x',t'}=\varphi\brac{x,t}\in\R\times\R_+$, such that $\alpha'$ and $\beta'$ do belong to $\MVDup$, where $\alpha'$ and $\beta'$ are the weak limits of the measures defined by
\begin{align*}
    &\bar{\alpha'}^N_{t'\brac{x,t}}\closedbrac{0,x'\brac{x,t}}=\bar{\alpha}^N_t\lt(-\infty,x\rt],
    &\bar{\beta'}^N_{t'\brac{x,t}}\closedbrac{0,x'\brac{x,t}}=\bar{\beta}^N_t\lt(-\infty,x\rt].
\end{align*}

The coordinate transformation $\brac{x',t'}=\varphi\brac{x,t}$ is based on the uniqueness of the path $g_{\brac{x,t}}$ due to Assumption \ref{ass:exist unique cond}; for any point $\brac{x,t}$, we assign the point $\brac{x',t'}$, given by
\begin{align*}
    \begin{cases}
    t'=t,\\
    x'=g_{\brac{x,t}}\brac{0}.
    \end{cases}
\end{align*}
The purpose of this transformation is to transform the aging trajectories in such a way that the priorities do not vary with time in the new system. 
The inverse transformation is given by
\begin{align*}
    \begin{cases}
    t=t',\\
    x=g_{\brac{x',0}}\brac{t'}.
    \end{cases}
\end{align*}
Note that $x'\geq 0$ because $g(t)$ is assumed to be decreasing.

It will be convenient to define a function $F:\MVD\to\MVD$ such that $\nu'=F\brac{\nu}$ is defined by
\begin{align*}
    \nu'_{t'}\closedbrac{0,x'}=\nu_{t\brac{x',t'}}\lt(-\infty,x\brac{x',t'}\rt].
\end{align*}

By the inverse transformation, the relation $\nu=F^{-1}\brac{\nu'}$ corresponds to the relation
\begin{align*}
    \nu_t\lt(-\infty,x\rt]=\nu'_{t'\brac{x,t}}\closedbrac{0,x'\brac{x,t}}.
\end{align*}

Note that $x'\mapsto\nu'_{t}\closedbrac{0,x'}$ is right-continuous for all $t>0$ because
\begin{align*}
    \abs{\nu'_t\closedbrac{0,x'+\eps}-\nu'_t\closedbrac{0,x'}}=\nu_t\lt(x\brac{x',t},x\brac{x'+\eps,t}\rt],
\end{align*}
and, by Theorem 6 in \cite{birkhoff1989},
\begin{align*}
    \abs{x\brac{x'+\eps,t}-x\brac{x',t}}=\abs{g_{\brac{x'+\eps,t}}\brac{t}-g_{\brac{x',t}}\brac{t}}\leq\eps e^{Lt}\to_\eps 0.
\end{align*}
In addition, it is monotone as $x'\mapsto x\mapsto\nu_t\lt(-\infty,x\rt]$ is.
We will show later that $\alpha'$ has indeed sample paths in $\MVDup$, and thus the fluid model can be defined through $\brac{\xi',\beta',\iota}=\Theta\brac{\alpha',\mu}$, and $\xi=F^{-1}\brac{\xi'}$, $\beta=F^{-1}\brac{\beta'}$.

Recall the model description from \S\ref{sec:Queueing model}.
Note that under Assumption \ref{ass:exist unique cond} the following conditions are equivalent:
\begin{align*}
    g_i^N\brac{t}\leq x=g_{\brac{x,t}}\brac{t} \iff {S'}_i^N:=g_i^N\brac{0}\leq g_{\brac{x,t}}\brac{0}=x'.
\end{align*}
With this, by Equations \eqref{eq:alphaN}-\eqref{eq:xiN}, $\alpha^N$, $\beta^N$ and $\xi^N$ take the form
\begin{align}
    &\alpha_t^N\lt(-\infty,x\rt] = \sum_{i=1}^\infty S_i^N\indicator{\tau_i^N\leq t}\indicator{{S'}_i^N\leq x'}={\alpha'}_{t'}^N\closedbrac{0,x'},\label{eq:alpha'N}\\
    &\beta_t^N\lt(-\infty,x\rt] = \sum_{i=1}^\infty S_i^N\indicator{\theta_i^N\leq t}\indicator{{S'}_i^N\leq x'}={\beta'}_{t'}^N\closedbrac{0,x'},\label{eq:beta'N}\\
    &\xi_t^N\lt(-\infty,x\rt]=\sum_{i=1}^\infty S_i^N\indicator{\tau_i^N\leq t< \theta_i^N}\indicator{{S'}_i^N\leq x'}={\xi'}_{t'}^N\closedbrac{0,x'}.\label{eq:xi'N}
\end{align}
Now, $t'\mapsto{\beta'}_{t'}^N\brac{x',\infty}$ is monotone and ${\alpha'}^N\in\MVDup$. 

Under Assumption \ref{ass:exist unique cond}, the priority condition is equivalent to the condition
\begin{align*}
    \begin{cases}
        \theta_i^N\geq\tau_j^N\\
        \theta_j^N\geq\tau_i^N\\
        g_i^N\brac{0}< g_j^N\brac{0}
    \end{cases}
    \Longrightarrow \theta_i^N<\theta_j^N.
\end{align*}
This implies
\begin{align*}
    \int{\xi'}_{t}^N\closedbrac{0,x'}d{\beta'}_{t}^N\brac{x',\infty}&=\sum_{j=1}^\infty{\xi'}_{\theta_j^N}^N\closedbrac{0,x'}\brac{{\beta'}_{\theta_j^N}^N\brac{x',\infty}-{\beta'}_{\theta_j^N-}^N\brac{x',\infty}}\\
    &=\sum_{j=1}^\infty\sum_{i=1}^\infty S_i^N\indicator{\tau_i^N\leq \theta_j^N< \theta_i^N}\indicator{g_i^N\brac{0}\leq x'}S_j^N\indicator{g_j^N\brac{0}>x'}\\
    &=0.
\end{align*}
Moreover, the non-idling property implies
\begin{align*}
    {\xi'}^N\closedbrac{0,x'}=0\quad d{\iota}^N\quad a.s.
\end{align*}
In addition
\begin{align*}
    {\xi'}_{t}^N\closedbrac{0,x'}={\alpha'}_{t}^N\closedbrac{0,x'}-{\beta'}_{t}^N\closedbrac{0,x'}={\alpha'}_{t}^N\closedbrac{0,x'}-\mu^N\brac{t}+\iota^N\brac{t}-J^N\brac{t}+J^N\brac{0}+{\beta'}_{t}^N\brac{x'\infty},
\end{align*}
and note that ${\alpha'}^N\in\MVDup$ and $\mu^N+J^N-J\brac{0}\in\Dup$. The later is justified by writing
\begin{align*}
    \mu^N+J^N-J\brac{0}=\brac{\mu^N-T^N}+\brac{T^N+J^N-J\brac{0}}=\iota^N+\beta^N\lt[0,\infty\rt).
\end{align*}
The above holds for the scaled processes as well, thus  $\brac{\bar{{\xi'}}^N,\bar{{\beta'}}^N,\bar{\iota}^N}=\Theta'\brac{\bar{{\alpha'}}^N,\bar{\mu}^N+\bar{J}^N-\bar{J}^N\brac{0}}$.

By Lemma \ref{lem:F continuous} and the continuous mapping theorem, it follows that $\bar{\alpha}^N\To\alpha$ if and only if $\bar{{\alpha'}}^N\To\alpha'=F\brac{\alpha}\in\MVC$.
In addition,
\begin{align*}
\alpha'_{t}\brac{\{x'\}}=\alpha'_{t}\closedbrac{0,x'}-\alpha'_{t}\lt[0,x'\rt)=\alpha_t\lt(-\infty,x\rt]-\alpha_t\brac{-\infty,x}=\alpha_t\brac{\{x\}}=0,
\end{align*}
so $\alpha'_{t'}$ is also atomless.
Finally, $\MVDup$ is a closed subset of $\MVD$ (\cite[Lemma 2.4]{atar18}), and each ${\alpha'}^N$ is in $\MVDup$, so we conclude that $\alpha'$ is  in $\MVCupnoatoms$.

We now have the same setting as in Theorem 5.13 from \cite{atar18}, by which $\brac{\bar{\xi'}^N,\bar{\beta'}^N,\bar{\iota}^N}\To\Theta'\brac{\alpha',\mu}$. Finally, by Lemma \ref{lem:F continuous}, we conclude $\bar{\xi}^N\To\xi=F^{-1}\brac{\xi'}$, $\bar{\beta}^N\To\beta=F^{-1}\brac{\beta'}$.
\end{proof}

\begin{lemma}\label{lem:F continuous}

\begin{enumerate}
    \item The functions $F$ and $F^{-1}$, restricted to $\MVD\brac{\closedbrac{0,T}},T>0$, are continuous.
    \item The function $F$, restricted to $\MVD\brac{\closedbrac{0,T}},T>0$, sends elements of $\MVC$ to $\MVC$.
\end{enumerate}

\end{lemma}
\begin{proof}
\begin{enumerate}
    \item Denote by $B\brac{\nu,\eps}$ the open ball in $\MVD\brac{\closedbrac{0,T}}$ with radius $\eps$, centered around $\nu$. 
We show that $F\brac{B\brac{\nu,\eps}}\subset B\brac{\nu',e^{LT}\eps}$, $\nu'=F\brac{\nu}$.
Take any $\lambda\in B\brac{\nu,\eps}$ and denote $\lambda'=F\brac{\lambda}$, then there exists a strictly increasing mapping $\gamma:\closedbrac{0,T}\to\closedbrac{0,T}$, with $\gamma\brac{0}=0$, $\gamma\brac{T}=T$ and $\sup_{t\in\closedbrac{0,T}}\abs{\gamma\brac{t}-t}<\eps$, such that for any $t\in\closedbrac{0,T}$ and any $x$
\begin{align*}
    \lambda_t\lt(-\infty,x-\eps\rt]-\eps\leq\nu_{\gamma\brac{t}}\lt(-\infty,x\rt]\leq\lambda_t\lt(-\infty,x+\eps\rt]+\eps,
\end{align*}
which is just $\sup_{t\in\closedbrac{0,T}}d_{\cal{L}}\brac{\nu_{\gamma\brac{t}},\lambda_t}<\eps$.

Making yet another use of Assumption \ref{ass:exist unique cond}, Theorem 6 in \cite{birkhoff1989} yields
\begin{align*}
    &x'\brac{x,t}-x'\brac{x-\eps,t}=g_{\brac{x,t}}\brac{0}-g_{\brac{x-\eps,t}}\brac{0}\leq \eps e^{Lt}\leq\eps e^{LT},\\
    &x'\brac{x+\eps,t}-x'\brac{x,t}=g_{\brac{x+\eps,t}}\brac{0}-g_{\brac{x,t}}\brac{0}\leq \eps e^{Lt}\leq\eps e^{LT}.
\end{align*}
Using the above inequalities, for any $t'\in\closedbrac{0,T}$ and $x$,
\begin{align*}
    &\nu'_{\gamma\brac{t}}\closedbrac{0,x'}=\nu_{\gamma\brac{t}}\lt(-\infty,x\rt]\leq\lambda_t\lt(-\infty,x+\eps\rt]+\eps=\lambda'_{t}\closedbrac{0,x'\brac{x+\eps,t}}+\eps<\lambda'_{t}\closedbrac{0,x'+\eps e^{LT}}+\eps e^{LT},\\
    &\nu'_{\gamma\brac{t}}\closedbrac{0,x'}=\nu_{\gamma\brac{t}}\lt(-\infty,x\rt]\geq\lambda_t\lt(-\infty,x-\eps\rt]-\eps=\lambda'_{t}\closedbrac{0,x'\brac{x-\eps,t}}-\eps>\lambda'_{t}\closedbrac{0,x'-\eps e^{LT}}-\eps e^{LT}.
\end{align*}
Hence, $\sup_{t'\in\closedbrac{0,T}}d_{\cal{L}}\brac{\nu'_{\gamma\brac{t}},\lambda'_{t}}<\eps e^{LT}$.

Continuity for $F^{-1}$ follows by similar arguments.
\item Just as in the first part of the Lemma, if for all $t>0$ and $x$
\begin{align*}
    \nu_{t+h}\closedbrac{0,x-\eps}-\eps \leq\nu_t\closedbrac{0,x}\leq\nu_{t+h}\closedbrac{0,x'+\eps }+\eps,
\end{align*}
then for all $t>0$ and $x'>0$
\begin{align*}
    \nu'_{t+h}\closedbrac{0,x'-\eps e^{LT}}-\eps e^{LT}\leq\nu'_t\closedbrac{0,x'}\leq\nu'_{t+h}\closedbrac{0,x'+\eps e^{LT}}+\eps e^{LT}.
\end{align*}
This means that for any $\eps>0$, choosing $h$ small enough such that $d_{\cal{L}}\brac{\nu_{t+h},\nu_t}<\delta$, $\delta=\eps e^{-LT}$ (such $h$ exists as we assume $\nu\in\MVC$),  implies $d_{\cal{L}}\brac{\nu'_{t+h},\nu'_t}<\eps$.   
\end{enumerate}
\end{proof}

\section{Examples} 

In this section we provide insight through simplifications and examples.
The examples assume some workload arrival distribution of a fluid system, and provide explicit solutions.
Recall that for a pair of data $\brac{\alpha,\mu}\in\MVCnoatoms\times\CCup$, the fluid solution is obtained by first applying the appropriate change of coordinates $x'=g_{\brac{x,t}}\brac{0}$ and then finding the measure valued process in $\MVCupnoatoms$, $\alpha'=F\brac{\alpha}$ (defined by $\alpha'_t\closedbrac{0,x'}=\alpha_t\lt(-\infty,x\rt]$).
Then one applies the MVSM $\brac{\xi',\beta',\iota}=\Theta\brac{\alpha',\mu}$ and finally $\xi=F^{-1}\brac{\xi'}$, $\beta=F^{-1}\brac{\beta'}$.
The entire procedure is summarized in the formulation of Theorem \ref{thm:Fluid limit} through Equations \eqref{eq:lim comp1} and \eqref{eq:lim comp2},
\begin{align*}
    &\beta_t\brac{x,\infty}+\iota\brac{t}=-\inf_{s\in\closedbrac{0,t}}\lt\{\brac{\Xi\brac{s,g_{\brac{x,t}}\brac{0}}-\mu\brac{s}}\land 0\rt\},\\
    &\xi_t\lt(-\infty,x\rt]=\Xi\brac{t,g_{\brac{x,t}}\brac{0}}-\mu\brac{t}-\inf_{s\in\closedbrac{0,t}}\lt\{\brac{\Xi\brac{s,g_{\brac{x,t}}\brac{0}}-\mu\brac{s}}\land 0\rt\},
\end{align*}
with $\Xi\brac{t,x}=\alpha_t\lt(-\infty,g_{\brac{x,0}}\brac{t}\rt]$.
In many cases, it is more natural to describe the arrival process through the instantaneous arrival distribution rather than through the cumulative arrival process $\alpha$.
The instantaneous arrival distribution describes the distribution of either the job sizes or the workload (each can be derived from the other) of the arriving work at each instant.
Therefore, it is useful to have a relation connecting $\alpha$ with this distribution.
Let $\pi_t$ be the instantaneous workload arrival distribution at time $t$, then $\alpha_t\lt(-\infty,x\rt]=\int_0^t\pi_s\closedbrac{0,g_{\brac{x,t}}\brac{s}}ds$.
All the following examples are of this form, where various instantaneous workload arrival distributions are considered.
First we present a simplification of the formula for $\xi$ and $\beta$ under certain conditions.
There appear four examples afterwards.
The first three examples are of linear aging and the last is of exponential aging.
In the first example, the instantaneous workload arrival distribution is uniform over the interval $[0,1]$, in the second example it is uniform over a time-varying interval, and in the two last examples it is Pareto distributed.

\subsection*{A simplification} 
Assume for simplicity that the arrival rate is larger than the service rate; in this case, $\iota=0$. 
We then may guess the following solution on the prime plane:
\begin{align}\label{eq:guess}
\begin{cases}
    \beta'_t\closedbrac{0,x'}=\alpha'_t\closedbrac{0,x'}\land\mu\brac{t},\\
    \xi'_t\closedbrac{0,x'}=\alpha'_t\closedbrac{0,x'}-\beta'_t\closedbrac{0,x'}=\closedbrac{\alpha'_t\closedbrac{0,x'}-\mu\brac{t}}^+.
    \end{cases}
\end{align}
Note that this is not always a valid guess, as it is not clear that 
\begin{align*}
    \beta'_t\brac{x',\infty}=\mu\brac{t}-\beta'_t\closedbrac{0,x'}=\closedbrac{\mu\brac{t}-\alpha'_t\closedbrac{0,x'}}^+
\end{align*}
is non-decreasing in $t$.
However, it is true in many cases, e.g. when the job size distribution, and the arrival and service rates are time invariant.
To verify that this guess is the true solution whenever $\beta'_t\brac{x',\infty}$ is non-decreasing, observe that either $\xi'_t\closedbrac{0,x'}$ or $\beta'_t\brac{x',\infty}$ is zero, 
Hence, $\int\xi'_t\closedbrac{0,x'}d\beta'_t\brac{x',\infty}=0$, and by definition $\beta'_t\lt[0,\infty\rt)=\mu\brac{t}$ and $\xi'\closedbrac{0,x'}=\alpha'\closedbrac{0,x'}-\beta'\closedbrac{0,x'}$.
Thus, this guess is the unique solution to the MVSP with data $\brac{\alpha',\mu}$.

The interpretation of the guess \eqref{eq:guess} is that there is a process $x^*\brac{t}=\inf\{x:\alpha'_t\closedbrac{0,x}\geq\mu\brac{t}\}$ such that the entire mass below $x^*\brac{t}$ was served by time $t$ and the entire mass above it is in the queue at time $t$. 
The guess \eqref{eq:guess} is equal to the fluid solution if $t\mapsto x^*\brac{t}$ is non-decreasing.
Note that the aging "pushes" $x^*\brac{t}$ to the right, thus does not invalidates the guess.

This simplification is used in the first example in this section to illustrate its usefulness.
We do not use this simplification in the other examples.

\subsection*{Linear aging with uniformly distributed workload} 
We begin by analyzing the case of uniformly distributed workload, specifically on $\closedbrac{0,x}$, with constant arrival rate. 
This corresponds for $\pi_s\closedbrac{0,x}=1\land\brac{x\lor 0}$.
The aging rule in this example is chosen to be the linear aging rule.
This is a rather simple example where closed form solutions can be found.
For linear aging, the cumulative workload arrival process is derived from $\pi$ by $\alpha_t\lt(-\infty,x\rt]=\int_0^t\pi_s\closedbrac{0,x+t-s}ds$.
We can compute
\begin{align*}
    \alpha_t\lt(-\infty,x\rt]&=\begin{cases}
        t &\text{ if }x>1,\\
        t+x-\frac{x^2}{2}-\frac{1}{2} &\text{ if }0<x<1\text{ and } x>1-t,\\
        xt+\frac{t^2}{2} &\text{ if }0<x<1\text{ and } x<1-t,\\
        x+t-\frac{1}{2} &\text{ if }x<0\text{ and } x>1-t,\\
        0 &\text{ if }x<0\text{ and } x<-t,\\
        \frac{\brac{x+t}^2}{2} &\text{ if }x<0\text{ and } -t<x<1-t,\\
    \end{cases}
\end{align*}
and, consequently, compute
\begin{align*}
    \alpha'_t\lt[0,x'\rt]&=\begin{cases}
        t &\text{ if }x'>1+t,\\
        x'-\frac{\brac{x'-t}^2}{2}-\frac{1}{2} &\text{ if }t<x'<1+t\text{ and } x'>1,\\
        x't-\frac{t^2}{2} &\text{ if }t<x'<t+1\text{ and } x'<1,\\
        x'-\frac{1}{2} &\text{ if }x'<t\text{ and } x'>1,\\
        0 &\text{ if }x'<t\text{ and } x'<0,\\
        \frac{x'^2}{2} &\text{ if }x'<t\text{ and } 0<x'<1.\\
    \end{cases}
\end{align*}
If we choose for example $\mu\brac{t}=t/2$, we can obtain $x^*\brac{t}=\brac{t+1}/2$, which is increasing with $t$. Hence, the guess \eqref{eq:guess} is valid. To simplify the expressions of $\xi'$ and $\beta'$, we consider times $t>1$. 
Then the solution is given by
\begin{align*}
    \beta'_t\lt[0,x'\rt]&=\begin{cases}
        \frac{t}{2} &\text{ if }x'>\frac{1+t}{2},\\
        x'-\frac{1}{2} &\text{ if }1<x'<\frac{t+1}{2},\\
        0 &\text{ if }x'<0,\\
        \frac{x'^2}{2} &\text{ if } 0<x'<1,\\
    \end{cases}\\
    \xi'_t\lt[0,x'\rt]&=\begin{cases}
        \frac{t}{2} &\text{ if }x'>1+t,\\
        x'-\frac{\brac{x'-t}^2}{2}-\frac{1}{2}-\frac{t}{2} &\text{ if }t<x'<1+t\text{ and } x'>1,\\
        x'-\frac{1+t}{2} &\text{ if }t>x'>\frac{t+1}{2},\\
        0 &\text{ if }x'<\frac{t+1}{2}.\\
    \end{cases}
\end{align*}

For the sake of completeness, we give the explicit expression for $\xi_t\lt(-\infty,x\rt]$,
\begin{align*}
    \xi_t\lt[0,x\rt]&=\begin{cases}
        \frac{t}{2} &\text{ if }x>1,\\
        x+t-\frac{x^2}{2}-\frac{1}{2}-\frac{t}{2} &\text{ if }0<x<1\text{ and } x>1-t,\\
        x+\frac{t-1}{2} &\text{ if }0>x>\frac{1-t}{2},\\
        0 &\text{ if }x<\frac{1-t}{2}.\\
    \end{cases}
\end{align*}

Figure \ref{fig:examples}(\subref{fig:Linear uniform}) shows the evolution of $\xi_t\lt(-\infty,x\rt]$ with time.
In this figure, and in all the other as well, the system is assumes to be empty at time $t=0$.

\begin{figure}
\begin{subfigure}{.5\textwidth}
  \centering
  \includegraphics[scale=0.7]{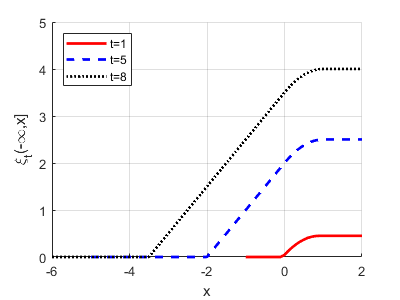}  
  \captionsetup{width=0.9\linewidth}
  \caption{$\xi_t\lt(-\infty,x\rt]$ when the aging is linear and the workload arrival distribution is uniform over the interval $[0,1]$.}
    \label{fig:Linear uniform}
\end{subfigure}
\begin{subfigure}{.5\textwidth}
  \centering
  \includegraphics[scale=0.7]{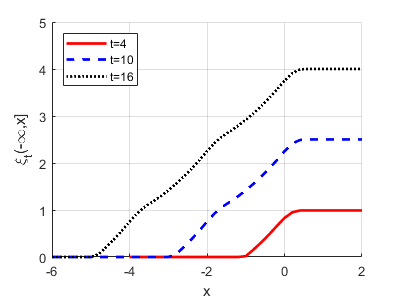} 
  \captionsetup{width=0.9\linewidth}
  \caption{$\xi_t\lt(-\infty,x\rt]$ when the aging is linear and the workload arrival distribution is uniform with time-varying support.}
    \label{fig:time varying supp}
\end{subfigure}
\begin{subfigure}{.5\textwidth}
  \centering
  \includegraphics[scale=0.7]{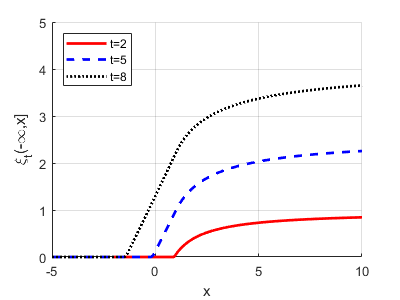}  
  \captionsetup{width=0.9\linewidth}
  \caption{$\xi_t\lt(-\infty,x\rt]$ when the aging is linear and the workload arrival distribution is Pareto with parameter $\eta=1.2$.}
    \label{fig:linear pareto eta}
\end{subfigure}
\begin{subfigure}{.5\textwidth}
  \centering
  \includegraphics[scale=0.7]{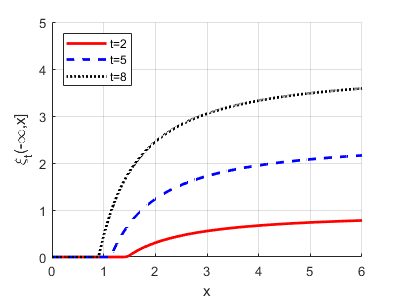}  
  \captionsetup{width=0.9\linewidth}
  \caption{$\xi_t\lt(-\infty,x\rt]$ when the aging is exponential and the workload arrival distribution is Pareto with parameters $\eta=1.2$ and $\lambda=0.1$.}
    \label{fig:exp pareto}
\end{subfigure}
\caption{Examples of $\xi_t\lt(-\infty,x\rt]$ for different aging rules and different workload arrival distribution.}
\label{fig:examples}
\end{figure}

\subsection*{Time varying arrival distribution} 
In many cases it is possible to obtain an expression for $\alpha_t\lt(-\infty,x\rt]$ when the instantaneous workload arrival distribution is periodic. 
This is exactly the setting in the following example.
In this example, the instantaneous workload distribution is again uniform, but its support is now a periodically-time-varying interval $\closedbrac{0,a(t)}$.
This corresponds to $\pi_t\closedbrac{0,x}=a(s)\land\brac{x\lor 0}$.
Our aging rule is still linear.
We solve for a specific choice of $a(t)$, a triangular wave- a piecewise linear periodic function as illustrated in Figure \ref{fig:WL distribution}.

\begin{figure}
    \centering
    \begin{tikzpicture}[scale=0.4]
    \draw[thick,->] (0,7)--(0,0)--(12,0);
    \draw[thick] (0,3)--(2,6)--(4,3)--(6,6)--(8,3)--(10,6)--(12,3);

\node[anchor=south] at (0,7) {$a(t)$};
\node[anchor=east] at (0,3) {$\frac{1}{2}$};
\node[anchor=east] at (0,6) {$1$};
\node[anchor=north] at (0,0) {$0$};
\node[anchor=north] at (2,0) {$1$};
\node[anchor=north] at (4,0) {$2$};
\node[anchor=north] at (6,0) {$3$};
\node[anchor=north] at (8,0) {$4$};
\node[anchor=north] at (10,0) {$5$};
\node[anchor=north west] at (12,0) {$t$};
\end{tikzpicture}
    \caption{Time-varying upper boundary of the workload distribution support.}
    \label{fig:WL distribution}
\end{figure}
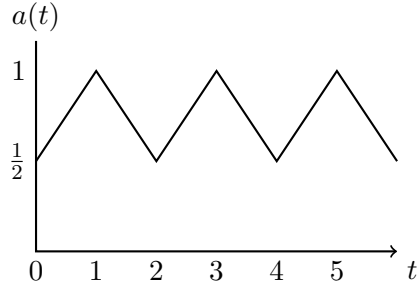

To describe the solution, we define the following.
\begin{align*}
    N\brac{x,t}&=2\floor{\frac{x+t}{2}},\\
    A_1\brac{x,t}&=\frac{1+\floor{x+t}}{2},\\
    s_1^*\brac{x,t}&=2\brac{x+t-A_1\brac{x,t}},\\
    A_2\brac{x,t}&=\frac{1}{2}-\floor{\frac{x+t}{2}},\\
    s_2^*\brac{x,t}&=\frac{2\brac{x+t-A_2\brac{x,t}}}{3}.
\end{align*}
To ease the notation, the dependency of $x$ and $t$ is omitted, and we simply write $N,A_i,s_i^*,i=1,2$. 
Let $D_1=\{\brac{x,t}:N\leq x+t<N+1/2\}$ and $D_2=\{\brac{x,t}:N+1/2\leq x+t<N+2\}$.
Solving the integral $\int_0^t\pi_s\closedbrac{0,x+t-s}ds$ gives
\begin{align*}
    &\alpha_t\lt(-\infty,x\rt]=\\
    &\begin{cases}
        \frac{3}{4}\lt\lfloor t\rt\rfloor+t-\lt\lfloor t\rt\rfloor-\frac{\brac{t-\lt\lfloor t\rt\rfloor}^2}{4} &\text{ if }t\leq s^*\text{ and }\lt\lfloor t\rt\rfloor\text{ is odd}\\
        \frac{3}{4}\lt\lfloor t\rt\rfloor+\frac{t-\lt\lfloor t\rt\rfloor}{2}+\frac{\brac{t-\lt\lfloor t\rt\rfloor}^2}{4} &\text{ if }t\leq s^*\text{ and }\lt\lfloor t\rt\rfloor\text{ is even}\\
        \frac{3}{4}\lt\lfloor s_1^*\rt\rfloor+s_1^*-\lt\lfloor s_1^*\rt\rfloor-\frac{\brac{s_1^*-\lt\lfloor s_1^*\rt\rfloor}^2}{4}+\brac{x+t}\brac{t-s_1^*}-\frac{t^2-s_1^{*2}}{2} &\text{ if }0\leq s_1^*<t\leq x+t\text{ and }D_1\\
        \frac{3}{4}\lt\lfloor s_2^*\rt\rfloor+\frac{s_2^*-\lt\lfloor s_2^*\rt\rfloor}{2}+\frac{\brac{s_2^*-\lt\lfloor s_2^*\rt\rfloor}^2}{4}+\brac{x+t}\brac{t-s_2^*}-\frac{t^2-s_2^{*2}}{2} &\text{ if }0\leq s_2^*<t\leq x+t\text{ and }D_2\\
        xt+\frac{t^2}{2} &\text{ if }s^*<0<t<x+t\\
        \frac{\brac{x+t}^2}{2} &\text{ if }s^*<0<x+t<t\\
        \frac{3}{4}\lt\lfloor s_1^*\rt\rfloor+s_1^*-\lt\lfloor s_1^*\rt\rfloor-\frac{\brac{s_1^*-\lt\lfloor s_1^*\rt\rfloor}^2}{4}+\brac{x+t}\brac{t+x-s_1^*}-\frac{\brac{t+x}^2-s_1^{*2}}{2} &\text{ if }0<s_1^*<x+t<t\text{ and }D_1\\
        \frac{3}{4}\lt\lfloor s_2^*\rt\rfloor+\frac{s_2^*-\lt\lfloor s_2^*\rt\rfloor}{2}+\frac{\brac{s_2^*-\lt\lfloor s_2^*\rt\rfloor}^2}{4}+\brac{x+t}\brac{t+x-s_2^*}-\frac{\brac{t+x}^2-s_2^{*2}}{2} &\text{ if }0<s_2^*<x+t<t\text{ and }D_2\\
        0 & \text{ if }x+t<0.
    \end{cases}
\end{align*}
In the above, the condition $s^*>t$ should be interpreted as either $s_1^*>t$ and $D_1$, or $s_2^*>t$ and $D_2$; and likewise for $s^*<0$.

It is now possible to find $\alpha'$ and $x^*(t)$, though obtaining $\xi'$, $\beta'$ and $\iota'$ is rather tedious even for such $\mu(t)$ for which the guess \eqref{eq:guess} is valid, and we do not continue with that line.
The general solution, as mentioned at the beginning of the section, takes the form
\begin{align*}
    \xi_t\lt(-\infty,x\rt]=\Xi\brac{t,x+t}-\mu\brac{t}-\inf_{s\in\closedbrac{0,t}}\lt\{\brac{\Xi\brac{s,x+t}-\mu\brac{s}}\land 0\rt\},
\end{align*}
with $\Xi\brac{t,x}=\alpha_t\lt(-\infty,x-t\rt]$.
Figure \ref{fig:examples}(\subref{fig:time varying supp}) shows the evolution of $\xi_t\lt(-\infty,x\rt]$ with time.
Note the difference between \ref{fig:examples}(\subref{fig:Linear uniform}) and \ref{fig:examples}(\subref{fig:time varying supp}).
The fact that $\pi_s$ is periodic is reflected in the wiggliness of the graph.

\subsection*{Pareto distributed workload} 
The Pareto distribution is a heavy tailed distribution, used often in queueing theory to model internet packet inter-arrival times (\cite{Mirtchev2008}).
This distribution is of a special interest in this section because we can find explicit expressions for $\alpha$ when the workload is Pareto distributed, both for linear and for exponential aging.
We begin with the linear aging setting, and compare the results later with the exponential aging setting.
For the Pareto distribution of the arriving workload we choose the parameters 1 and $\eta$, which correspond to $\pi_s\closedbrac{0,x}=\brac{1-x^{-\eta}}\indicator{x\geq 1}$.
In this case we can calculate $\alpha$  and obtain
\begin{align*}
    \alpha_t\closedbrac{0,x}&=\int_0^t\pi_s\closedbrac{0,x+t-s}ds=\int_0^t\brac{1-\brac{x+t-s}^{-\eta}}\indicator{x+t-s\geq 1}ds\\
    &=
    \begin{cases}
    0 & x<1-t,\\
    x+t-1+\frac{\brac{x+t}^{1-\eta}-1}{\eta} & 1-t\leq x\leq 1,\\
    t+\frac{\brac{x+t}^{1-\eta}-x^{1-\eta}}{\eta} & 1< x. 
    \end{cases}
\end{align*}
Now one can then find
\begin{align*}
    \alpha'_t\closedbrac{0,x'}&=
    \begin{cases}
    0 & x'<1,\\
    x'-1+\frac{x'^{1-\eta}-1}{\eta} & 1\leq x'\leq 1+t,\\
    t+\frac{x'^{1-\eta}-\brac{x'-t}^{1-\eta}}{\eta} & 1+t< x'. 
    \end{cases}
\end{align*}
The solution takes again the form
\begin{align*}
    \xi_t\lt(-\infty,x\rt]=\Xi\brac{t,x+t}-\mu\brac{t}-\inf_{s\in\closedbrac{0,t}}\lt\{\brac{\Xi\brac{s,x+t}-\mu\brac{s}}\land 0\rt\},
\end{align*}
with $\Xi\brac{t,x}=\alpha_t\lt(-\infty,x-t\rt]$.

Figure \ref{fig:examples}(\subref{fig:linear pareto eta}) shows the evolution of $\xi_t\lt(-\infty,x\rt]$ with time, when we used $\eta=1.2$.
Figure \ref{fig:linear pareto t} shows the influence of the Pareto parameter $\eta$ on the shape of the graphs. 

\begin{figure}
 \centering
    \includegraphics[scale=0.7]{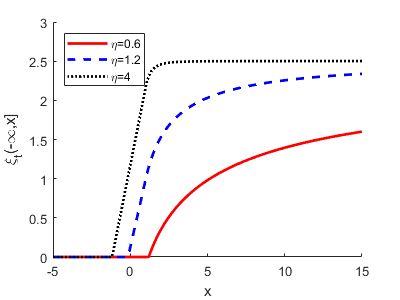}
    \caption{$\xi_t\lt(-\infty,x\rt]$ when the aging is linear and the workload arrival distribution is Pareto, after 5 time units.}
    \label{fig:linear pareto t}
\end{figure}

\subsection*{Exponential aging} 
We continue with an example in a setting with exponential aging.
For this example we choose again the Pareto distribution with parameters $1$ and $\eta$ as the instantaneous workload arrival distribution corresponding to $\pi_s\closedbrac{0,x}=\brac{1-x^{-\eta}}\indicator{x\geq 1}$.
We calculate
\begin{align*}
    &\alpha_t\closedbrac{0,x}=\int_0^t\pi_s\closedbrac{0,x e^{-\lambda\brac{s-t}}}ds=\int_0^{t+\lambda^{-1}\log\brac{x\land 1}}\brac{1-x^{-\eta}e^{\lambda\eta \brac{s-t}}}ds\\
    &\hspace{9cm}=
    \begin{cases}
    0 & x\leq e^{-\lambda t}\\
    t+\frac{1}{\lambda}\log(x)-\frac{x^{-\eta}}{\lambda\eta}\brac{x^{\eta}-e^{-\lambda\eta t}} &e^{-\lambda t}<x< 1,\\
    t-\frac{x^{-\eta}}{\lambda\eta}\brac{1-e^{-\lambda\eta t}} &x\geq 1,
    \end{cases}
    \\
    &\alpha'_t\closedbrac{0,x'}=\alpha_t\closedbrac{0,x'e^{-\lambda t}}=
    \begin{cases}
    0 & x'\leq 1,\\
    \frac{1}{\lambda}\log(x')-\frac{1-x'^{-\eta}}{\lambda\eta} &1<x'<e^{\lambda t},\\
    t-\frac{x'^{-\eta}}{\lambda\eta}\brac{e^{\lambda\eta t}-1} &x'\geq e^{\lambda t}.
    \end{cases}
\end{align*}
The solution takes the form
\begin{align*}
    \xi_t\lt(-\infty,x\rt]=\Xi\brac{t,xe^{\lambda t}}-\mu\brac{t}-\inf_{s\in\closedbrac{0,t}}\lt\{\brac{\Xi\brac{s,xe^{\lambda t}}-\mu\brac{s}}\land 0\rt\},
\end{align*}
with $\Xi\brac{t,x}=\alpha_t\lt(-\infty,xe^{-\lambda t}\rt]$.

Figure \ref{fig:examples}(\subref{fig:exp pareto}) shows the evolution of $\xi_t\lt(-\infty,x\rt]$ with time for this example with the choice $\mu(t)=t/2$.
It is interesting to compare \ref{fig:examples}(\subref{fig:linear pareto eta}) with \ref{fig:examples}(\subref{fig:exp pareto}), both sharing the same workload arrival distribution while differing in their aging rules.


\bibliographystyle{is-abbrv}
\bibliography{Bib}

\end{document}